\documentclass[a4paper,12pt]{article}
\usepackage{amsmath,amsthm,amssymb,amsfonts}
\usepackage{bbm,bm}
\usepackage{latexsym}
\usepackage{mathrsfs}
\usepackage{threeparttable}
\usepackage{tabularx}
\usepackage{booktabs}
\usepackage{graphicx,subfigure}
\usepackage{color}
\usepackage{indentfirst}
\usepackage{geometry}
\usepackage{caption}
\usepackage{lineno}
\usepackage{abstract}
\usepackage{enumerate,mdwlist}
\usepackage[numbers,sort&compress]{natbib}
\usepackage[colorlinks,
            linkcolor=blue, 
            citecolor=red  
            ]{hyperref}
\usepackage{epstopdf}

\makeatletter
\@addtoreset{equation}{section}

\newtheorem{thm}{Theorem}[section]
\newtheorem{prop}[thm]{Proposition}

\newtheorem{lem}[thm]{Lemma}
\newtheorem{cor}[thm]{Corollary}
\newtheorem{ex}[thm]{Example}
\newtheorem{pb}[thm]{Problem}

\newtheorem{conj}[thm]{Conjecture}
\newtheorem{remark}[thm]{Remark}

\begin{document}

\setlength{\baselineskip}{20pt}
\underline{}\begin{center}{\Large \bf Some tight bounds on the minimum and maximum forcing numbers of graphs}\footnote{This work is supported by NSFC\,(Grant No. 11871256).}

\vspace{4mm}

{Qianqian Liu, Heping Zhang \footnote{The corresponding author.}
\renewcommand\thefootnote{}\footnote{E-mail addresses: liuqq2016@lzu.edu.cn(Q.Liu), zhanghp@lzu.edu.cn(H.Zhang).}}

\vspace{2mm}

\footnotesize{ School of Mathematics and Statistics, Lanzhou University, Lanzhou, Gansu 730000, P. R. China}
\end{center}
\noindent {\bf Abstract}: Let $G$ be a simple graph with $2n$ vertices and a perfect matching.
We denote by $f(G)$ and $F(G)$ the minimum and maximum forcing number of $G$, respectively.

Hetyei obtained that the maximum number of edges of graphs $G$ with a unique perfect matching is $n^2$. We know that $G$ has a unique
perfect matching if and only if $f(G)=0$.
Along this line, we generalize the classical result to all graphs $G$ with $f(G)=k$ for $0\leq k\leq n-1$,
and characterize corresponding extremal graphs as well. Hence we get a non-trivial lower bound of $f(G)$ in terms of the order and size. For bipartite graphs, we gain corresponding stronger results. Further, we obtain a new upper bound of $F(G)$.
For bipartite graphs $G$, Che and Chen (2013) obtained that $f(G)=n-1$ if and only if $G$ is complete bipartite graph $K_{n,n}$. We completely characterize  all bipartite graphs $G$ with $f(G)= n-2$.

\vspace{2mm} \noindent{\it Keywords}: Perfect matching; Minimum forcing number; Maximum forcing number; Bipartite graph
\vspace{2mm}


 {\setcounter{section}{0}
\section{\normalsize Introduction}\setcounter{equation}{0}
 We consider only finite and simple graphs. Let $G$ be a graph with vertex set $V(G)$ and edge set $E(G)$. The \emph{order} of $G$ is the number of vertices in  $G$, and the \emph{size} of $G$, written $e(G)$, is the number of edges in $G$.

\emph{A perfect matching} $M$ of a graph $G$ is a set of disjoint edges covering all vertices of $G$.
A subset $S\subseteq M$ is called a \emph{forcing set} of $M$ if $S$ is not contained in any other perfect matching of $G$. The smallest cardinality of a forcing set of $M$ is called the \emph{forcing number} of $M$, denoted by $f(G,M)$. The concept was originally introduced by Harary et al. \cite{2} and by Klein and Randi\'{c} \cite{3}, which plays an important role in resonance theory.

For a perfect matching $M$ of $G$,
a cycle of $G$ is \emph{M-alternating} if its edges appear alternately in $M$ and $E(G)\setminus M$. Clearly, $M$ is a unique perfect matching of $G$ if and only if
$G$ contains no $M$-alternating cycles.
\begin{lem}\cite{6}\label{l1} Let $G$ be a graph with a perfect matching $M$. Then $S\subseteq M$ is a forcing set of $M$ if and only if $S$ contains at least one edge of every $M$-alternating cycle.
\end{lem}
Let $C(G,M)$ denote the maximum number of disjoint $M$-alternating cycles in $G$. Then $f(G,M)\geq C(G,M)$ by Lemma \ref{l1}.
For plane bipartite graphs, Pachter and Kim pointed out the following minimax theorem.
\begin{thm}\cite{6}\label{thm6} Let $G$ be a plane bipartite graph. Then $f(G, M)=C(G,M)$ for any perfect matching $M$ of $G$.
\end{thm}
For a  vertex subset $T$ of $G$, we write $G-T$ for the subgraph of $G$ obtained by deleting all vertices in $T$ and their incident edges. Sometimes, we write $G[V(G)\setminus T]$ for the subgraph $G-T$, \emph{induced by} $V(G)\setminus T$. If $T=\{v\}$, we write $G-v$ rather than $G-\{v\}$.


Let $G$ and $H$ be bipartite graphs. We say $G$ \emph{contains} $H$ if $G$ has a subgraph $L$ such that $G-V(L)$ has a perfect matching
and $L$ is isomorphic to an even subdivision of $H$. In \cite{LL1} and some articles related to matching theory, $G$ \emph{contains} $H$ is also called $H$ is a \emph{conformal minor} of $G$.
Guenin and Thomas obtained the following general minimax result in
somewhat different manner (see Corollary 5.8 in \cite{54}).

\begin{thm}\cite{54} Let $G$ be a bipartite graph with a perfect matching $M$. Then $G$ has no $K_{3,3}$ or the Heawood graph as a conformal minor if and only if
$f(G', M')=C(G', M')$ for each subgraph $G'$ of $G$ such that $M'=M\cap E(G')$ is a perfect matching in $G'$.
\end{thm}

The \emph{minimum} and \emph{maximum forcing number} of $G$ are the minimum and maximum values of $f(G,M)$ over all perfect matchings $M$ of $G$, denoted by $f(G)$ and $F(G)$, respectively. The \emph{degree} of a vertex $v$ in $G$, written $d_G(v)$, is the number of edges incident to $v$. A \emph{pendant vertex} of $G$ is a vertex of degree 1. We denote by $\delta(G)$ and $\Delta(G)$ the minimum and maximum degrees of the vertices of $G$. The problem of finding the minimum forcing number of bipartite graphs with the maximum degree 4 is NP-complete \cite{5}.

The path and cycle with $n$ vertices are denoted by $P_n$ and $C_n$, respectively. The \emph{cartesian product} of graphs $G$ and $H$, written $G\times H$.
Pachter and Kim \cite{6} showed that $f(P_{2n}\times P_{2n})=n$ and $F(P_{2n}\times P_{2n})=n^2$. Riddle \cite{7} got that $f(C_{2m}\times C_{2n})=2\text{min}\{m,n\}$, and
Kleinerman \cite{16} obtained that $F(C_{2m}\times C_{2n})=mn$. Afshani et al. \cite{5} obtained that
$F(P_{2k}\times C_{2n})=kn$ and $F(P_{2k+1}\times C_{2n})=kn+1$, and they \cite{5} proposed a problem: what is the maximum forcing number of non-bipartite graph $P_{2m} \times C_{2n+1}$? Jiang and Zhang \cite{29} solved the problem and obtained that $F(P_{2m}\times C_{2n+1})=m(n+1)$.
For any $k$-regular bipartite graph $G$ with $n$ vertices in each partite set, Adams et al. \cite{4} showed that $F(G)\geq (1-\frac{log~(2e)}{log~k})n$, where $e$ is the base of the natural logarithm. Hence, for hypercube $Q_k$ where $k\geq 2$, $F(Q_k)>c2^{k-1}$ for any constant $0<c<1$ and sufficient large $k$ (see \cite{7}). Diwan \cite{25} proved that $f(Q_k)=2^{k-2}$ by linear algebra for $k\geq 2$, which solved a conjecture proposed by Pachter and Kim \cite{6}.
For hexagonal systems, Xu et al. \cite{27} proved that the maximum forcing number is equal to its resonant number. For polyomino graphs \cite{42,LW17} and BN-fullerene graphs \cite{40}, the same result also holds.
For more researches on the minimum and maximum forcing numbers, see \cite{1,29,56,62}.

For graphs with a unique perfect matching, there are some classical  results. To describe these results, we define a bipartite graph $H_{n,k}$ of order $2n$ as follows, where $n$ and $k$ are integers with $0\leq k\leq n-1$:
The bipartition of $H_{n,k}$ is $U\cup V$, where $U=\{u_1,u_2,\dots,u_n\}$ and $V=\{v_1,v_2,\dots,v_n\}$, such that $u_iv_j\notin E(H_{n,k})$ if and only if $1\leq i <j\leq n-k$ (see $H_{6,2}$ in Fig. \ref{fe4}). It is clear that $$d_{H_{n,k}}(u_l)=k+l \text{ and } d_{H_{n,k}}(v_l)=n-l+1 \text{ for } l=1,2,\dots,n-k$$ and the other vertices have degree $n$. So $H_{n,n-1}$ is isomorphic to $K_{n,n}$, which is the complete bipartite graph with each partite set having $n$ vertices.

Let $\hat{H}_{n,0}$ be the graph obtained by adding all possible edges in $V$ to $H_{n,0}$ (see $\hat{H}_{5,0}$ in Fig. \ref{fe4}).
Obviously, $\{u_iv_i|i=1,2,\dots,n\}$ is the unique perfect matching of $H_{n,0}$ and $\hat{H}_{n,0}$.
\begin{figure}[h]
\centering
\includegraphics[height=3.2cm,width=16cm]{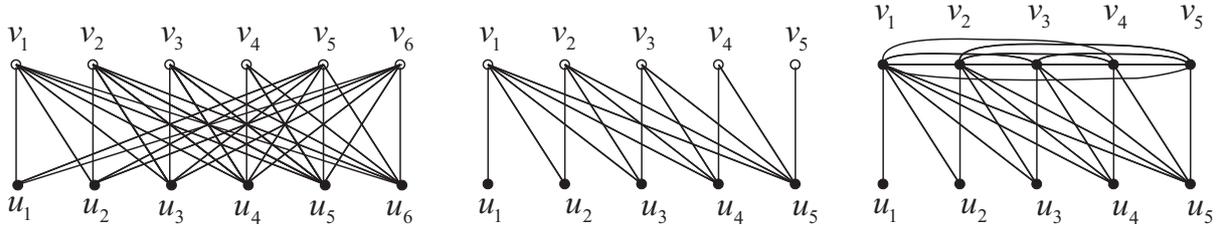}
\caption{\label{fe4} The graphs $H_{6,2}$, $H_{5,0}$ and $\hat{H}_{5,0}$.}
\end{figure}
A graph is \emph{split} if its vertex set can be partitioned into a clique and an independent set. Since $U$ is an independent set and $V$ is a clique of $\hat{H}_{n,0}$, $\hat{H}_{n,0}$ is a split graph. A graph is called a \emph{cograph} if it is either a singleton or it can be obtained by the disjoint union or join of two cographs, where the \emph{join} of two graphs $G$ and $H$, written $G\vee H$, is formed by taking the disjoint union of these two graphs and additionally adding the edges $\{xy|x\in V(G),y\in V(H)\}$.

For (bipartite) graphs with a unique perfect matching, there are some classical results (see Lemma 4.3.2 in \cite{14} for bipartite graphs, and Corollary 1.6 in \cite{46} or Corollary 5.3.14 in \cite{14} for general graphs).
\begin{thm}\cite{14}\label{l3} Let $G$ be a bipartite graph of order $2n$ and with a unique perfect matching. Then $G$ has two pendant vertices lying in different partite sets and $e(G)\leq\frac{n(n+1)}{2}$. Moreover, equality holds if and only if $G$ is $H_{n,0}$.
\end{thm}

\begin{thm}\label{cl4} \cite{46,14} Let $G$ be a graph of order $2n$ and with a unique perfect matching. Then $e(G)\leq n^2$, and equality holds if and only if $G$ is $\hat{H}_{n,0}$.
\end{thm}

We assume that the graphs $G$ in question have $2n$ vertices and a perfect matching. Then $0\leq f(G)\leq F(G)\leq n-1$.
If we use the terminology of forcing number, then $G$ has a unique perfect matching if and only if $f(G)=0$.
Along this line, we generalize Theorems \ref{l3} and \ref{cl4} to all graphs $G$ with $f(G)=k$ for $0\leq k\leq n-1$ in Section 2.
 In detail, we show that $e(G)\leq n^2+2nk-k^2-k$ and characterize corresponding extremal graphs. In turn, we obtain that $f(G)\geq n-\frac{1}{2}-\sqrt{2n^2-n-e(G)+\frac{1}{4}}$. For bipartite graphs, both bounds can be improved to $\frac{(n-k)(n+k+1)}{2}+nk$ and $n-\frac{1}{2}-\sqrt{2n^2-2e(G)+\frac{1}{4}}$, respectively.
For some special graphs, we give another lower bound of $f(G)$ in terms of $\delta(G)$. Precisely, if $G$ is a bipartite graph then $f(G)\geq \delta(G)-1$, and if $G$ is a split graph or a cograph then $f(G)\geq \frac{\delta(G)-1}{2}$.
In Section 3, we consider all graphs $G$ with $F(G)=k$ for $0\leq k\leq n-1$ and get that $e(G)\geq \frac{n(n+1)}{n-k}-k-1$. As a result, we obtain a new upper bound of $F(G)$ and compare it with two known bounds derived from the maximum anti-forcing numbers. A bipartite graph $G$ has $f(G)=n-1$ if and only if $G$ is $K_{n,n}$. In Section 4, we determined all bipartite graphs $G$ with $f(G)= n-2$.
\section{\normalsize Some lower bounds of the minimum forcing number}
In this section, we generalize Theorems \ref{l3} and \ref{cl4} to all bipartite and general graphs $G$ of order $2n$ and with $f(G)=k$ for $0\leq k\leq n-1$, respectively. By these results, we obtain two non-trivial lower bounds of $f(G)$ with respect to the order and size. For some special classes of graphs $G$, we also give a lower bound of $f(G)$ by using $\delta(G)$.
For a subset $S$ of $E(G)$, we use $V(S)$ to denote the set of all end-vertices of edges in $S$.

\begin{thm}\label{p3}Let $G$ be a graph of order $2n$ and with $f(G)=k$ for $0\leq k\leq n-1$. Then
\begin{eqnarray}
e(G)\leq n^2+2nk-k^2-k,\label{e4}
\end{eqnarray}
and equality holds if and only if $G$ is $\hat{H}_{n-k,0}\vee K_{2k}$ where $K_{2k}$ denotes the complete graph of order $2k$.
\end{thm}
\begin{proof}Suppose to the contrary that $e(G)\geq n^2+2nk-k^2-k+1$. Let $M$ be any perfect matching of $G$ and $S$ be any subset of $M$ with size no less than $n-k$. We are to prove that $M\setminus S$ is not a forcing set of $M$.
If we have done, then $f(G,M)\geq k+1$. By the arbitrariness of $M$, we acquire that $f(G)\geq k+1$, a contradiction.

Since $e(K_{2n})=2n^2-n$, we have
\begin{eqnarray}
n^2+2nk-k^2-k+1=e(K_{2n})-(n-k)(n-k-1)+1, \label{e8}
\end{eqnarray}
$$e(G[V(S)])\geq \binom{2|S|}{2}-[(n-k)(n-k-1)-1]=2|S|^2-|S|-(n-k)(n-k-1)+1.$$ So $e(G[V(S)])-(|S|^2+1)\geq |S|(|S|-1)-(n-k)(n-k-1)\geq 0$ for $x^2-x$ is monotonically increasing in $[\frac{1}{2},+\infty)$ and $|S|\geq n-k\geq 1$. Thus, $e(G[V(S)])\geq |S|^2+1$. By Theorem \ref{cl4}, $G[V(S)]$ has at least two perfect matchings. That is, $M\setminus S$ is not a forcing set of $M$.

Suppose that $G$ is the join of $\hat{H}_{n-k,0}$ and $K_{2k}$. By Theorem \ref{cl4}, $e(\hat{H}_{n-k,0})=(n-k)^2$. Since exactly two vertices in $\hat{H}_{n-k,0}$ may be not adjacent in $G$, we get that  $$e(G)=\binom{2n}{2}-[~\binom{2(n-k)}{2}-(n-k)^2~]=n^2+2nk-k^2-k.$$

Conversely, suppose that equality in (\ref{e4}) holds. Since $f(G)=k$, there exists a perfect matching $M$ of $G$ and a  minimum forcing set $S$ of $M$ such that $|S|=f(G,M)=f(G)$. By Lemma \ref{l1}, $G[V(M\setminus S)]$ contains no $M$-alternating cycles. Since (\ref{e8}) holds, we have $$e(G[V(M\setminus S)])\geq \binom{2(n-k)}{2}-(n-k)(n-k-1)=(n-k)^2.$$ By Theorem \ref{cl4}, $e(G[V(M\setminus S)])=(n-k)^2$ and $G[V(M\setminus S)]$ is  $\hat{H}_{n-k,0}$. Furthermore, each vertex in $V(S)$ is adjacent to all other vertices in $G$. So we have $G=\hat{H}_{n-k,0}\vee K_{2k}$.
\end{proof}

By inversing (\ref{e4}), we obtain a general lower bound on $f(G)$.
\begin{cor}\label{cp3} Let $G$ be a graph of order $2n$ and with a perfect matching. Then
\begin{eqnarray}
f(G)\geq n-\frac{1}{2}-\sqrt{2n^2-n-e(G)+\frac{1}{4}}, \label{e6}
\end{eqnarray}
and equality holds if and only if $G$ is $\hat{H}_{n-k,0} \vee K_{2k}$.
\end{cor}
\begin{proof}Let $f(G)=k$. Then $0\leq k\leq n-1$. By Theorem \ref{p3}, $e(G)\leq n^2+2nk-k^2-k$. That is, $k^2-(2n-1)k-n^2+e(G)\leq 0$.
So \begin{eqnarray}
n-\frac{1}{2}-\sqrt{2n^2-n-e(G)+\frac{1}{4}}\leq k\leq n-\frac{1}{2}+\sqrt{2n^2-n-e(G)+\frac{1}{4}}. \label{e12}
\end{eqnarray}
Hence (\ref{e6}) holds.

Since $n-\frac{1}{2}+\sqrt{2n^2-n-e(G)+\frac{1}{4}}\geq n$ and $n-1$ is a trivial upper bound of $f(G)$, the second inequality in (\ref{e12}) holds. So equality in (\ref{e6}) holds if and only if equality in (\ref{e4}) holds.
Hence these graphs such that two equalities in (\ref{e4}) and (\ref{e6}) hold are the same.
\end{proof}
For bipartite graphs, we can obtain corresponding stronger results than Theorem \ref{p3} and Corollary \ref{cp3}.

\begin{thm}\label{p6}Let $G=(U,V)$ be a bipartite graph of order $2n$ and with $f(G)=k$ for $0\leq k\leq n-1$. Then
\begin{eqnarray}
e(G)\leq \frac{(n-k)(n+k+1)}{2}+nk,\label{e5}
\end{eqnarray}
and equality holds if and only if $G$ is $H_{n,k}$.
\end{thm}
\begin{proof}Suppose to the contrary that $e(G)\geq \frac{(n-k)(n+k+1)}{2}+nk+1$.
Let $M$ and $S$ be defined as that in the proof of Theorem \ref{p3}. By the same arguments, we will prove that $M\setminus S$ is not a forcing set of $M$.
Since $e(K_{n,n})=n^2$ and
\begin{eqnarray}
\frac{(n-k)(n+k+1)}{2}+nk+1=e(K_{n,n})-\frac{(n-k)(n-k-1)}{2}+1, \label{e9}
\end{eqnarray} we have $e(G[V(S)])\geq |S|^2-[\frac{(n-k)(n-k-1)}{2}-1]$. So $$e(G[V(S)])-[\frac{|S|(|S|+1)}{2}+1]\geq \frac{1}{2}[|S|^2-|S|-(n-k)^2+n-k]\geq 0$$ for $x^2-x$ is strictly monotonic  increasing in $[\frac{1}{2},+\infty)$ and $|S|\geq n-k\geq 1$. Therefore, $e(G[V(S)])\geq \frac{|S|(|S|+1)}{2}+1$. By Theorem \ref{l3}, $G[V(S)]$ has at least two perfect matchings. That is,  $M\setminus S$ is not a forcing set of $M$.

Suppose that  $G$ is $H_{n,k}$. Let $G'=G[\{u_i,v_i|i=1,2,\dots,n-k\}]$. Then $G'$ is isomorphic to $H_{n-k,0}$. By Theorem \ref{l3}, $e(G')=\frac{(n-k)(n-k+1)}{2}$. Since each vertex of $V(G)\setminus V(G')$ has vertex $n$, $u_iv_j\notin E(G)$ if and only if $1\leq i<j\leq n-k$ if and only if $u_iv_j\notin E(G')$. Thus,  $$e(G)=n^2-[(n-k)^2-e(G')]=\frac{(n-k)(n+k+1)}{2}+nk.$$

Conversely, suppose that equality in (\ref{e5}) holds. Since $f(G)=k$, there exists a perfect matching $M$ of $G$ and a minimum forcing set $S$ of $M$ such that $|S|=f(G,M)=f(G)$. By Lemma \ref{l1}, $G[V(M\setminus S)]$ has a unique perfect matching. Since (\ref{e9}) holds, we have $$e(G[V(M\setminus S)])\geq (n-k)^2-\frac{(n-k)(n-k-1)}{2}=\frac{(n-k)(n-k+1)}{2}.$$ By Theorem \ref{l3}, we obtain that $e(G[V(M\setminus S)])=\frac{(n-k)(n-k+1)}{2}$ and
$G[V(M\setminus S)]$ is $H_{n-k,0}$. Let $u_i$ and $v_j$ be two vertices of  $U\cap V(S)$ and $V\cap V(S)$, respectively. Then $u_i$ is adjacent to all vertices of $V$ and $v_j$ is adjacent to all vertices of $U$. So $G$ is $H_{n,k}$.
\end{proof}
By inversing (\ref{e5}), we obtain a lower bound on $f(G)$ for bipartite graphs.
\begin{cor}\label{cp6} Let $G$ be a bipartite graph of order $2n$ and with a perfect matching. Then
\begin{eqnarray}
f(G)\geq n-\frac{1}{2}-\sqrt{2n^2-2e(G)+\frac{1}{4}},\label{e7}
\end{eqnarray}
and  equality holds if and only if $G$ is $H_{n,k}$.
\end{cor}
\begin{proof}Let $f(G)=k$. Then $0\leq k\leq n-1$. By Theorem \ref{p6}, $e(G)\leq \frac{(n-k)(n+k+1)}{2}+nk$. That is to say, $k^2-(2n-1)k-n^2-n+2e(G)\leq 0$.
So
\begin{eqnarray}
n-\frac{1}{2}-\sqrt{2n^2-2e(G)+\frac{1}{4}}\leq k\leq n-\frac{1}{2}+\sqrt{2n^2-2e(G)+\frac{1}{4}}.\label{e13}
\end{eqnarray}
Consequently, (\ref{e7}) holds.

Since $n-\frac{1}{2}+\sqrt{2n^2-2e(G)+\frac{1}{4}}\geq n$ and $n-1$ is a trivial upper bound of $f(G)$, the second inequality in (\ref{e13}) holds. So equality in (\ref{e7}) holds if and only if equality in (\ref{e5}) holds. Hence the graphs such that two equalities in (\ref{e5}) and (\ref{e7}) hold are the same.
\end{proof}
\begin{remark}\label{rm1}{\rm
The right sides in (\ref{e6}) and (\ref{e7})
are strictly monotonic increasing about $e(G)$. Hence the bounds in (\ref{e6}) and (\ref{e7}) are effective respectively for graphs $G$ with $e(G)\geq n^2$ and $e(G)\geq \frac{1}{2}(n^2+n)$.}
\end{remark}

In the sequel, we will give some lower bounds of $f(G)$ in terms of $\delta(G)$.

\begin{thm}\label{thm2}If $G$ is a bipartite graph with a perfect matching, then $f(G)\geq \delta(G)-1$. Moreover, the bound is tight.
\end{thm}
\begin{proof}Let $M$ be a perfect matching of $G$ and $S$ be a minimum forcing set of $M$ such that $|S|=f(G,M)=f(G)$. By Lemma  \ref{l1}, $G-V(S)$ has a unique perfect matching. By Theorem \ref{l3}, $G-V(S)$ has a pendant vertex, say $u$. Then all but one of the neighbors of $u$ are incident with edges in $S$. Combining that $G$ is a bipartite graph, we obtain that $f(G)=|S|\geq d_G(u)-1\geq\delta(G)-1$.

Note that $H_{n,k}$ is a bipartite graph with $\delta(H_{n,k})=k+1$. Since equality in (\ref{e5}) holds for $H_{n,k}$, we have $f(H_{n,k})=k=\delta(H_{n,k})-1$. Thus the bound is tight.
\end{proof}
For a graph $G$ of order $2n$, we say a set $U =\{u_1,u_2,\dots, u_n\}$ \emph{forces a unique perfect matching} in $G$ if $u_i$ is a pendant vertex of $G_i$ whose only neighbor is $v_i$ for every $1\leq i\leq n$, where $G_1=G$, $G_i=G_{i-1}-\{u_{i-1},v_{i-1}\}$ for $2\leq i\leq n$. Clearly, if $U$ forces a unique perfect matching in $G$, then $\{u_iv_i|i=1,2,\dots,n\}$ is a unique perfect matching of $G$.
\vspace{2mm}

 For cographs and split graphs, Chaplick et al. \cite{45} obtained the
following result.
\begin{lem}\label{l9}\cite{45} If $G$ is a cograph or a split graph, then $G$ has a unique perfect matching if and only if some set forces a unique
perfect matching in $G$.
\end{lem} Lemma \ref{l9} guarantees the following result.
\begin{thm}\label{thm3}
If $G$ is a split graph or a cograph with a perfect matching,
then $f(G)\geq \frac{\delta(G)-1}{2}$. Moreover, the bound is tight.
\end{thm}
\begin{proof}
Let $M$ and $S$ be defined as that in the proof of Theorem \ref{thm2}. By Lemma \ref{l1}, $G-V(S)$ has a unique perfect matching. Since $G-V(S)$ is still a split graph or a cograph, $G-V(S)$ has a pendant vertex by Lemma \ref{l9}, say $u$. Then all but one of the neighbors are incident with edges in $S$. Hence we have $2|S|\geq d_G(u)-1\geq\delta(G)-1$. So $f(G)=|S|\geq \frac{\delta(G)-1}{2}$.

Next we will show that this bound is tight. Let $G_1=\hat{H}_{n-k,0}\vee K_{2k}$ where $0\leq k \leq n-1$. Since $V(G_1)$ can be partitioned into an independent set $I=\{u_1,u_2,\dots,u_{n-k}\}$ and a clique $V(G_1)\setminus I$,
$G_1$ is a split graph with $\delta(G_1)=2k+1$. Combining that equality in (\ref{e4}) holds for $G_1$, we obtain that $f(G_1)=k=\frac{\delta(G_1)-1}{2}$.

Let $G_2=(n-k)K_2\vee K_{2k}$ where $0\leq k \leq n-1$ and $(n-k)K_2$ denotes $(n-k)$ disjoint copies of $K_2$. Since $(n-k)K_2$ and $K_{2k}$ are two cographs, $G_2$ is a cograph with $\delta(G_2)=2k+1$. By Theorem \ref{thm3}, we have $f(G_2)\geq \frac{\delta(G_2)-1}{2}=k$. Let $M$ be a perfect matching of $G_2$ consisting of $(n-k)K_2$ and a perfect matching $M_1$ of $K_{2k}$. Then $M_1$ is a forcing set of $M$. So $f(G_2)\leq f(G_2,M)\leq |M_1|=k$. Thus, $f(G_2)=k=\frac{\delta(G_2)-1}{2}$.
\end{proof}

\begin{remark}{\rm Theorem \ref{thm3} is not necessarily true for general graphs.

Suppose that $G_3=H\vee K_{2(n-4)}$ where $H$ is shown in Fig. \ref{fe55} and $n\geq 4$. Assume that the vertices of $K_{2(n-4)}$ is $\{u_{i},v_i|i=5,6,\dots,n\}$.

Let $M=M_1\cup M_2$ be a perfect matching of $G_3$ where $M_1=\{u_1v_1,u_2v_2,u_3v_3,u_4v_4\}$ is a perfect matching of $H$ and $M_2$ is that of $K_{2(n-4)}$. Then $M_2\cup\{u_4v_4\}$ is a forcing set of
$M$ since $H-\{u_4,v_4\}$ has a unique perfect matching.
So $f(G_3)\leq f(G_3,M)\leq n-3$.
But $\delta(G_3)=2(n-4)+4=2n-4$ and $\frac{\delta(G_3)-1}{2}=n-\frac{5}{2}>n-3\geq f(G_3)$.}
\end{remark}
\begin{figure}[h]
\centering
\includegraphics[height=3.5cm,width=15cm]{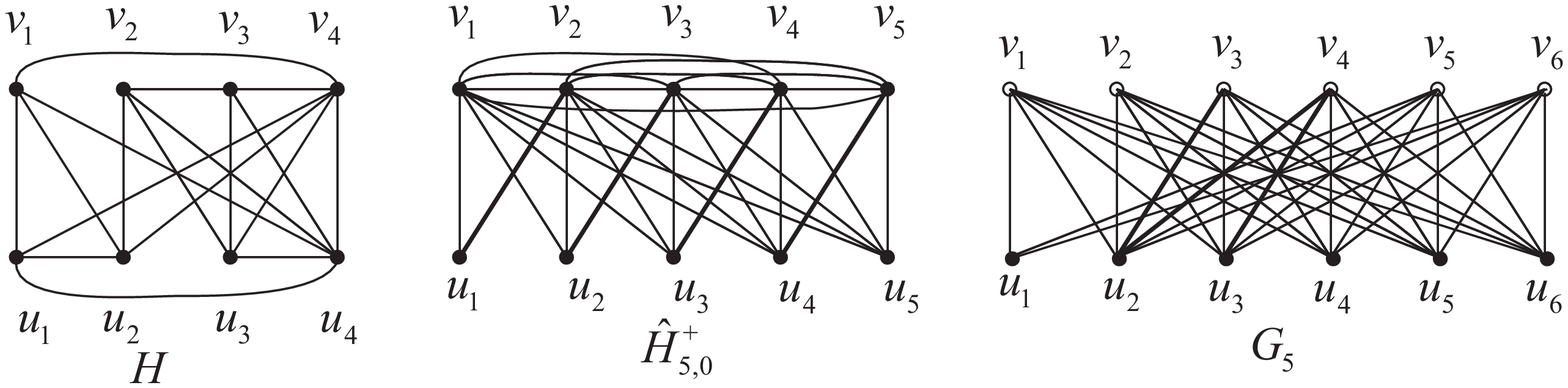}
\caption{\label{fe55} Graph $H$, $\hat{H}_{5,0}^+$ where $n=5, k=0$, and $G_5$ where $n=6$ and $k=i=2$.}
\end{figure}

Using these lower bounds obtained, we can calculate the minimum forcing numbers of some graphs which are not extremal graphs of corresponding minimum forcing numbers.
\begin{ex}Let $G_4=\hat{H}_{n-k,0}^+\vee
 K_{2k}$ where $\hat{H}_{n-k,0}^+$ is a graph obtained from $\hat{H}_{n-k,0}$ by adding a set of edges $T=\{u_iv_{i+1}|i=1,2,\dots,n-k-1\}$ for some $0\leq k\leq n-2$ (see $\hat{H}_{5,0}^+$ in Fig. \ref{fe55}). Then $f(G_4)=k+1$.
\end{ex}

 \begin{proof}By Remark \ref{rm1}, $f$ is strictly monotonic increasing about the number of edges. Combining Corollary \ref{cp3} and $|T|\geq 1$, we have $$f(G_4)\geq n-\frac{1}{2}-\sqrt{2n^2-n-e(G_4)+\frac{1}{4}}>  n-\frac{1}{2}-\sqrt{2n^2-n-e(\hat{H}_{n-k,0}\vee K_{2k})+\frac{1}{4}}=k$$ as $\hat{H}_{n-k,0}\vee K_{2k}$ is the extremal graph of Theorem \ref{p3}.
 So $f(G_4)\geq k+1$.

On the other hand, let $M=T\cup\{u_{n-k}v_1\}\cup  \{u_{i}v_{i}|i=n-k+1,n-k+2,\dots,n\}$ be a perfect matching of $G_4$. Since
$\{u_1,u_2,\dots,u_{n-k-1}\}$ forces a unique perfect matching in $G_4[V(T)]$,  $G_4[V(T)]$ has a unique perfect matching. So $M\setminus T$ is a forcing set of $M$ and $f(G_4)\leq f(G_4,M)\leq |M\setminus T|=k+1$.
\end{proof}

\begin{ex}\label{exa1}Let $G_5=(U,V)$ be a bipartite graph obtained from $H_{n,k}$ by
adding a set of edges $\{u_iv_{n-k-1},u_iv_{n-k},u_{i+1}v_{n-k}\}$ for some
 $1\leq i\leq n-k-2$ (an example $G_5$ shown in Fig. \ref{fe55}). Then $f(G_5)=k+2$.
\end{ex}
\begin{proof}On one hand, let $M_0=\{u_iv_i|i=1,2,\dots,n\}$ be a perfect matching of $G_5$ and $S_0=\{u_iv_i|i=n-k-1,n-k,\dots,n\}$. Since   $\{u_i|i=1,2,\dots,n-k-2\}$ forces a unique perfect matching in $G_5-V(S_0)$,  $G_5-V(S_0)$ has a unique perfect matching. So $S_0$ is a forcing set of $M_0$ and $f(G_5)\leq f(G_5,M_0)\leq |S_0|=k+2$.

On the other hand, if $i=1$, then $\delta(G_5)=k+3$. By Theorem \ref{thm2},  $f(G_5)\geq k+2$. Suppose that  $i\geq2$ and $M$ is a perfect matching of $G_5$. Let   $L=\{u_jv_{l_j}|j=1,2,\dots,i-1\}$ be a subset of $M$. Then $1\leq l_j\leq j$ or $l_j\geq n-k+1$.

Let $G_5'=G_5-V(L)$ and $M'=M\setminus L$. Then $G_5'$ is a bipartite graph with bipartition $U'\cup V'$, where $U'=U\cap V(G_5')$ and $V'=V\cap V(G_5')$. Next we will prove that $\delta(G_5')\geq k+3$.
Since $d_{G_5}(u_i)=d_{G_5}(u_{i+1})=k+i+2,~d_{G_5}(v_{n-k-1})=d_{G_5}(v_{n-k})=k+3$ and other vertices have same degree as in $H_{n,k}$. Combining that $|V\setminus V'|=i-1$ we have $$d_{G_5'}(u_j)\geq d_{G_5}(u_i)-(i-1)=k+i+2-(i-1)\geq k+3 \text{ for } j=i,i+1,\dots,n \text{ and}$$  $$d_{G_5'}(v_{j})=d_{G_5}(v_{j})\geq d_{G_5}(v_{n-k})=k+3 \text{ for } j=i,i+1,\dots,n-k.$$
For  $1\leq j\leq i-1$ and $v_j\in V'$, we have $$d_{G_5'}(v_{j})=d_{G_5}(v_{j})-(i-j)=(n-j+1)-(i-j)=n-i+1\geq k+3$$ and other vertices of $G_5'$ have degree $n-(i-1)\geq k+3$. Thus $\delta(G_5')\geq k+3$.

By Theorem \ref{thm2}, $f(G_5',M')\geq f(G_5')\geq k+2$. By definition of forcing sets, we have $f(G_5,M)\geq f(G_5',M')\geq k+2$. By the arbitrariness of $M$, $f(G_5)\geq k+2$.
\end{proof}

\section{\normalsize Some upper bounds of the maximum forcing number}
Let $G$ be a graph with a perfect matching. Lei et al. \cite{43} obtained that $F(G)$ is no more than the maximum anti-forcing number of $G$. Hence, we can derive two upper bounds of $F(G)$ from those of the maximum anti-forcing number.

The anti-forcing number of a graph was introduced by Vuki$\check{c}$evi$\acute{c}$ and Trinajsti$\acute{c}$ \cite{51} as the smallest number of edges whose
removal results in a subgraph with a unique perfect matching. Recently, Lei et al. \cite{43} defined the \emph{anti-forcing number} of a perfect matching $M$ of $G$ as the minimal number of edges not in $M$ whose removal to make $M$ as a single perfect matching of the resulting graph.
The \emph{maximum anti-forcing number} of $G$, denoted by $Af(G)$, is the maximum value of anti-forcing numbers over all perfect matchings of $G$.

For a connected graph $G$, the \emph{cyclomatic number} of it is defined as $r(G)=|E(G)|-|V(G)|+1$.
Deng and Zhang \cite{52} obtained that $Af(G)\leq r(G)$. Afterwards, Shi and Zhang \cite{11} gave a new bound $Af(G)\leq \frac{2|E(G)|-|V(G)|}{4}$.
By these, we obtain the following result.
\begin{cor}\label{c1} Let $G$ be a connected graph of order $2n$ and with a perfect matching. Then
\begin{equation*}
 F(G)\leq
 \begin{cases}
 \frac{e(G)-n}{2}, & \quad {\text{if $e(G)\geq3n-2$}};\\
 e(G)-2n+1,&\quad {\text{otherwise}}.
 \end{cases}
 \end{equation*}
\end{cor}

In this section, we will characterize all graphs $G$ with $F(G)=\frac{e(G)-n}{2}$. But we have not been able to characterize the other yet.  Furthermore, we would give a new upper bound on $F(G)$ and obtain that the new bound is better than Corollary \ref{c1} for graphs $G$ with a larger number of edges.

Given $S,T\subseteq V(G)$, we write $E(S,T)$ for the set of edges having one end-vertex in $S$ and the other in $T$ and $e(S,T)$ for the number of edges in $E(S,T)$.

\begin{prop}\label{p4}Let $G$ be a graph of order $2n$ and with a perfect matching. Then $F(G)\leq\frac{e(G)-n}{2}$, and equality holds if and only if $G$ consists of $\frac{e(G)-n}{2}$ cycles of length 4 and $2n-e(G)\geq 0$ independent edges.
\end{prop}
\begin{proof}It suffices to prove the second part.
If $G$ consists of $\frac{e(G)-n}{2}$ cycles of length 4 and $2n-e(G)$ independent edges, then $G$ is a plane bipartite graph and has exactly
$\frac{e(G)-n}{2}$ $M$-alternating cycles for any perfect matching $M$ of $G$. By Theorem \ref{thm6}, $f(G,M)=\frac{e(G)-n}{2}$. So $F(G)=\frac{e(G)-n}{2}$.

Conversely, if $F(G)=\frac{e(G)-n}{2}$, then there exists a perfect matching $M$ of $G$ and a minimum forcing set $S$ of $M$ such that $|S|=f(G,M)=\frac{e(G)-n}{2}$. By Lemma \ref{l1}, we have $G[V(M\setminus S)]$ contains no $M$-alternating cycles. But $S\setminus \{e\}$ is not a forcing set of $M$ for any edge $e$ of $S$ by the minimality of $S$.  By Lemma \ref{l1}, $G[V((M\setminus S)\cup \{e\})]$ contains an $M$-alternating cycle $C_e$. So $e$ is contained in $C_e$ and
\begin{eqnarray}
e(V(M\setminus S),V(S))=\sum_{e\in S}e(V(M\setminus S),V(e))\geq 2|S|\label{ell3}.
\end{eqnarray}

 Since $e(G)\geq e(V(M\setminus S),V(S))+|M|\geq 2|S|+n=e(G)$, we obtain that all equalities hold. Thus $e(V(M\setminus S),V(S))= 2|S|$, and both $G[V(S)]$ and $G[V(M\setminus S)]$ consist of independent edges. By equality (\ref{ell3}),
 $e(V(M\setminus S),V(e))=2$ for each edge $e$ of $S$. So $C_e$ is an $M$-alternating 4-cycle.

Moreover, $C_{e_1}\cap C_{e_2}=\emptyset$ for any pair of distinct edges $e_1$ and $e_2$ of $S$. Otherwise, there exist two edges $e_1$ and $e_2$ of $S$ so that $e'\in E(C_{e_1})\cap E(C_{e_2})$ for some edge $e'$ of $M\setminus S$. Then $E(V(M\setminus (S\cup \{e'\}),V(\{e_1,e_2\}))=\emptyset$. Thus $(S\setminus \{e_1,e_2\})\cup \{e'\}$ is a forcing set of $M$ with size less than $S$, a contradiction. Hence $|M\setminus S|\geq |S|$, which implies $e(G)\leq 2n$.
 Therefore, $G$ consists of $\frac{e(G)-n}{2}$ cycles of length 4 and $2n-e(G)$ independent edges.
\end{proof}

Next we will give a new upper bound of $F(G)$ and we need a lemma as follows.
\begin{lem}\label{c2}Let $G$ be a graph of order $2n$ and with $f(G,M)=k$ for $0\leq k\leq n-1$. Then there exists an edge $uv\in M$ such that $d_G(u)+d_G(v)\geq \frac{2n}{n-k}$. If equality holds, then $(n-k)\mid n$.
\end{lem}
\begin{proof}Let $S$, $e$ and $C_e$ be defined as in the proof of necessity of Proposition \ref{p4}. Then $e$ is contained in an $M$-alternating cycle $C_e$ and (\ref{ell3}) holds.

Let $d_G(u)+d_G(v)=$ max$\{d_G(x)+d_G(y)|xy\in M\setminus S\}$. Then
\begin{eqnarray}
(n-k)[d_G(u)+d_G(v)]&\geq& \sum_{xy\in M\setminus S}[d_G(x)+d_G(y)]\label{e1}\\
                   &=&2e(G[V(M\setminus S)])+e(V(M\setminus S),V(S))\notag\\
                   &\geq&2(n-k)+2k\label{e2}\\
                   &=&2n.\notag
\end{eqnarray}So we obtain the required result.

If $d_G(u)+d_G(v)=\frac{2n}{n-k}$, then all equalities in (\ref{ell3})-(\ref{e2}) hold. So $e(V(M\setminus S),V(e))=2$ for each edge $e$ of $S$,
$d_G(u)+d_G(v)=d_G(x)+d_G(y)$ for each edge $xy$ of $M\setminus S$, and $G[V(M\setminus S)]$ consists of $n-k$ independent edges. Thus $C_e$ is a cycle of length 4, and $d_G(x)=d_G(y)$ for each edge $xy$ of $M\setminus S$. So $2n=(n-k)[d_G(u)+d_G(v)]=2(n-k)d_G(u)$ and $(n-k)\mid n$.
\end{proof}

\begin{thm}\label{p7}Let $G$ be a graph of order $2n$ and with $F(G)=k$ for $0\leq k\leq n-1$. Then
\begin{eqnarray}
e(G)\geq \frac{n(n+1)}{n-k}-k-1.\label{e3}
\end{eqnarray}
\end{thm}
\begin{proof}We proceed by induction on $n$. For $n=1$, we have $F(G)=k=0$ and $e(G)=1$. So (\ref{e3}) holds. Suppose that  $n\geq 2$. If $k=0$, then $G$ has a unique perfect matching and $e(G)\geq n$.
 Next we suppose that  $1\leq k\leq n-1$.

Since $F(G)=k$, there exists a perfect matching $M$ of $G$ such that $f(G,M)=k$. By Lemma \ref{c2}, there exists an edge $uv\in M$ such that $d_G(u)+d_G(v)\geq \frac{2n}{n-k}$. Let $G'=G-\{u,v\}$. Then $F(G')\geq k-1$. Suppose to the contrary that $F(G')\leq k-2$. Then $M'=M\setminus\{uv\}$ is a perfect matching of $G'$ and $f(G',M')\leq k-2$. Let $S'$ be a minimum forcing set of $M'$. Then $|S'|=f(G',M')$. By Lemma \ref{l1}, $G'-V(S')$ has a unique perfect matching. Combining that $G-V(S'\cup\{uv\})=G-\{u,v\}-V(S')=G'-V(S')$, we obtain that $S'\cup\{uv\}$ is a forcing set of $M$. So $f(G,M)\leq |S'\cup\{uv\}|\leq k-1$, which is a contradiction. Therefore, $k-1\leq F(G')\leq k$.

If $F(G')=k-1\leq n-2$, then $e(G')\geq \frac{(n-1)n}{n-1-(k-1)}-(k-1)-1=\frac{(n-1)n}{n-k}-k$ by the induction hypothesis. By Lemma \ref{c2}, we get that  $$e(G)=e(G')+d_G(u)+d_G(v)-1\geq\frac{(n-1)n}{n-k}-k+\frac{2n}{n-k}-1
=\frac{n^2+n}{n-k}-k-1.$$ Otherwise, we have $F(G')=k$. Since $G'$ has
$2(n-1)$ vertices, we have $F(G')\leq n-2$. By the induction hypothesis,
 $e(G')\geq \frac{(n-1)n}{n-1-k}-k-1$.
By Lemma \ref{c2} and  $1\leq k\leq n-1< 2n-1$, 
\begin{eqnarray*}
e(G)-(\frac{n^2+n}{n-k}-k-1)&=&[e(G')+d_G(u)+d_G(v)-1]-(\frac{n^2+n}{n-k}-k-1)\\
&\geq&[\frac{(n-1)n}{n-1-k}-k-1+\frac{2n}{n-k}-1]-(\frac{n^2+n}{n-k}-k-1)\\
&=&\frac{(n-1)n}{n-1-k}+\frac{k-n^2}{n-k}\\
&=&\frac{k(2n-1-k)}{(n-1-k)(n-k)}>0.
\end{eqnarray*}
Hence (\ref{e3}) holds and we complete the proof.
\end{proof}
By inversing (\ref{e3}), we obtain an upper bound of $F(G)$. 
\begin{cor}\label{cp4} Let $G$ be a graph of order $2n$ and with a perfect matching. Then
\begin{eqnarray}
F(G)\leq\frac{\sqrt{e^2(G)+2(n+1)e(G)-3n^2-2n+1}}{2}-\frac{e(G)+1-n}{2} \label{ie1}.
\end{eqnarray}
\end{cor}
\begin{proof}Let $F(G)=k$. Then $0\leq k\leq n-1$. By Theorem \ref{p7}, $e(G)\geq \frac{n^2+n}{n-k}-k-1$. That is, $k^2+k(e(G)+1-n)-ne(G)+n^2\leq 0$.
By solving the quadratic inequality of $k$, we obtain that
  $(\ref{ie1})$ holds.
\end{proof}

Note that $nK_2$ and $K_{n,n}$ are two graphs such that equalities in (\ref{e3}) and (\ref{ie1}) hold. So the bounds in Theorem \ref{p7} and Corollary \ref{cp4} are tight.

\begin{remark}{\rm
By a simple calculation, we obtain that the upper bound in Corollary \ref{cp4} is less than $\frac{e(G)-n}{2}$
when $e(G)>\frac{7n-2}{3}$ and less than $r(G)$ when  $e(G)>2n-1+\frac{\sqrt{2n^2-2n}}{2}$. Hence for connected graphs $G$ of order  $2n~(n\geq 2)$, the upper bound in Corollary \ref{cp4} is less than that of Corollary \ref{c1} when  $e(G)>2n-1+\frac{\sqrt{2n^2-2n}}{2}$.}
\end{remark}

\section{\normalsize Characterization of bipartite graphs $G$ of order $2n$ and with $f(G)=n-2$}
Che and Chen \cite{1} asked a question: how to characterize the graphs $G$ of order $2n$ and with $f(G)=n-1$. For bipartite graphs, they obtained the following result.
\begin{thm}\cite{10}\label{thm1} Let $G$ be a bipartite graph of order $2n$. Then $f(G)=n-1$ if and only if $G$ is complete bipartite graph $K_{n,n}$.
\end{thm}

The present authors have obtained the following result for general graphs.
\begin{thm}\cite{47}\label{thm7} Let $G$ be a graph of order $2n$. Then  $f(G)=n-1$ if and only if $G$ is a complete multipartite graph with each partite set having size no more than $n$ or $G$ is a graph obtained by adding arbitrary additional edges in the same partite set to $K_{n,n}$.
\end{thm}

In this section, we will determine all bipartite graphs $G$ of order $2n$ and with $f(G)=n-2$ for $n\geq 2$.
For an edge subset $S$ of $G$, we write $G-S$ for the subgraph of $G$ obtained by deleting the edges in $S$.
Let $F_0$ be a bipartite graph which contains exactly one edge and each partite set has exactly two vertices. A bipartite graph $G$ is \emph{$F_0$-free} (resp. $P_4$-free) if it contains no induced subgraph isomorphic to $F_0$ (resp. $P_4$), where the two partite sets of the induced subgraph have the same sizes.
\begin{lem}\label{l6}Let $G=(U,V)$ be a bipartite graph. Then $G$ is $F_0$-free if and only if
  $G$ can be obtained from $K_{|U|,|V|}$ by deleting all edges of some disjoint complete bipartite subgraphs.
 \end{lem}
 \begin{proof}Sufficiency. For a pair of vertices $u\in U$ and $v\in V$, we have $uv\notin E(G)$ if and only if $u$ and $v$ lie in the same complete bipartite subgraph deleted edges of $K_{|U|,|V|}$. Suppose to the contrary that $G$ contains an induced subgraph $H$ isomorphic to $F_0$. Without loss of generality, we may suppose that $V(H)=\{u_1,u_2,v_1,v_2\}$ and $u_1v_1$ is the edge of $H$. Then these three pairs of vertices $\{u_1,v_2\}$, $\{u_2,v_1\}$ and $\{u_2,v_2\}$ are in the same complete bipartite subgraphs deleted edges of $K_{|U|,|V|}$, respectively. Hence the four vertices $u_1,u_2,v_1$ and $v_2$ lie in the same complete bipartite subgraph deleted edges of $K_{|U|,|V|}$, which contradicts that $u_1v_1$ is an edge of $G$.

 Necessity. Let $G'=K_{|U|,|V|}-E(G)$. Then $G$ and $G'$ are bipartite spanning subgraphs of $K_{|U|,|V|}$. It is obvious that $G$ is $F_0$-free if and only if $G'$ is $P_4$-free. It suffices to prove that every component of $G'$ with at least two vertices is a complete bipartite graph, and let $C'$ be such a component with bipartition $\{u_1,u_2,\dots,u_i\}\cup\{v_1,v_2,\dots,v_j\}$.

We will proceed by induction on $|V(C')|$. If $i=1$ or $j=1$, then we have done. So let $i\geq 2$ and $j\geq 2$.
Then there exists a vertex $x$ of $C'$ such that $C'-x$ is connected. This is verified by choosing $x$ as an end-vertex of a longest path of $C'$. Without loss of generality, we may assume that $x=u_i$. Since $C'-u_i$ is $P_4$-free,  $C'-u_i$ is isomorphic to $K_{i-1,j}$ by the induction hypothesis. Since $C'$ is connected, there exists $1\leq k\leq j$ such that $u_iv_k\in E(C')$. Since $\{u_1v_k,u_1v_{l}\}\subseteq E(C')$ for any $1\leq l\leq j$ and $l\neq k$ and $C'$ is $P_4$-free, we obtain that $u_iv_{l}\in E(C')$. So $C'$ is a complete bipartite graph $K_{i,j}$.
\end{proof}

If $G$ is a graph obtained from $K_{n,n}$ by deleting all edges of some disjoint complete bipartite subgraphs, then we call these disjoint complete bipartite subgraphs \emph{deleted subgraphs of $G$}. Naturally, we assume that each deleted subgraph contains at least one vertex of each partite set of $K_{n,n}$. Also, we say that a graph is obtained from $K_{n,n}$ by such operations, we mean that the graph is not $K_{n,n}$.

The independence number of $G$ is denoted by $\alpha(G)$. An equivalent condition of bipartite graphs with a perfect matching is given below. (see Exercise 3.1.40 in \cite{36}).
\begin{lem}\cite{36}\label{ll3} Let $G$ be a bipartite graph of order $2n$. Then $\alpha(G)=n$ if and only if $G$ has a perfect matching.
 \end{lem}

An edge $e$ of $G$ is \emph{allowed} if it lies in some perfect matching of $G$ and \emph{forbidden} otherwise.
A graph is said to be \emph{elementary} if its allowed edges form a connected subgraph. Hetyei obtained the following result (see Theorem 1 in \cite{L15}).
\begin{lem}\label{ll4}\cite{L15} A bipartite graph is elementary if and only if it is connected and every edge is allowed.
\end{lem}

Let $\mathcal{G}_1$ be the set of all graphs obtained from $K_{n,n}$ by deleting all edges of some disjoint complete bipartite subgraphs and the orders of its deleted subgraphs are no more than $n$, and $\mathcal{G}_2$ be the set of all bipartite graphs of order $2n$ consisting of two complete bipartite graphs with perfect matchings and some forbidden edges between them (see Fig. \ref{fpp111}).
\begin{figure}[h]
\centering
\includegraphics[height=2.5cm,width=15cm]{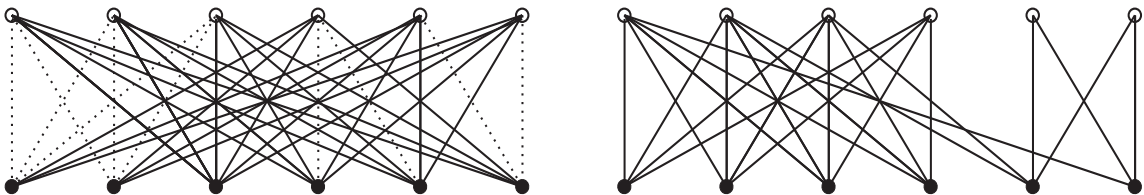}
\caption{\label{fpp111}A graph in $\mathcal{G}_1$ and a graph in $\mathcal{G}_2$.}
\end{figure}

\begin{thm}\label{thm4}Let $G$ be a bipartite graph of order $2n$ for $n\geq 2$. Then $f(G)=n-2$ if and only if $G$ is a graph in $\mathcal{G}_1$ or $\mathcal{G}_2$.
\end{thm}
\begin{proof}Sufficiency. First we prove that a graph $G\in\mathcal{G}_1$ has a perfect matching. Since the orders of deleted subgraphs of $G$ are no more than $n$, we have $\alpha(G)=n$. By Lemma \ref{ll3}, $G$ has a perfect matching.
 Let $G$ be a graph in $\mathcal{G}_1$ or $\mathcal{G}_2$. By Theorem \ref{thm1}, $f(G)\leq n-2$ since $G$ is not $K_{n,n}$. Next we will prove that $f(G)\geq n-2$.

Suppose that $G$ is a graph in $\mathcal{G}_1$. For $n=2$, we have $f(G)=0$ and the theorem holds. Let $n\geq 3$. Suppose to the contrary that $f(G)\leq n-3$. Then there exists a perfect matching $M$ of $G$ and a minimum forcing set $S$ of $M$ such that $|S|=f(G,M)=f(G)$. By Lemma \ref{l1}, $G-V(S)$ has a unique perfect matching.
So there are three distinct edges $\{e_1,e_2,e_3\}\subseteq M\setminus S$ such that $G[V(\{e_1,e_2,e_3\})]$ has a unique perfect matching. Set $e_i=u_iv_i$ for $1\leq i\leq 3$. By Theorem \ref{l3}, $G[V(\{e_1,e_2,e_3\})]$ contains two pendant vertices and we may assume such two vertices are $u_1$ and $v_3$.
Then $G[\{u_1,v_2,u_2,v_3\}]$ is isomorphic to $F_0$, which contradicts Lemma \ref{l6}.

Suppose that $G$ is a graph in $\mathcal{G}_2$. We denote by $G_1$ and $G_2$ the two complete bipartite subgraphs of $G$ with perfect matchings. Then $M\cap E(G_i)$ is a perfect matching of $G_i$ for any perfect matching $M$ of $G$ where $i\in \{1,2\}$.
 For any subset $S$ of $M$ such that $|S|\leq n-3$, $G-V(S)$ contains three edges of $M$ and two of them lie in some complete bipartite subgraph, say $G_1$. Then $G-V(S)$ contains an $M$-alternating cycle in $G_1$. By Lemma \ref{l1}, $S$ is not a forcing set of $M$. Thus, $f(G,M)\geq n-2$. By the arbitrariness of $M$, we have $f(G)\geq n-2$.

 Necessity. Since $f(G)=n-2$, $G$ has a perfect matching and each partite set has $n$ vertices. By Theorem \ref{thm1}, $G$ is not $K_{n,n}$. If $G$ is $F_0$-free, then $G$ is a graph obtained from $K_{n,n}$ by deleting all edges of some disjoint complete bipartite subgraphs by Lemma \ref{l6}. Since $G$ has a perfect matching, the orders of its deleted subgraphs are no more than $n$. So $G$ is a graph in $\mathcal{G}_1$. If $G$ is not $F_0$-free, then $G$ contains an induced subgraph $H$ isomorphic to $F_0$ and $n\geq 3$. We claim that the edge $e$ of $H$ is a forbidden edge in $G$. Otherwise, there exists a perfect matching $M$ of $G$ containing $e$. Let $\{e,e',e''\}$ be the three distinct edges of $M$ incident with the vertices of $H$. Then $G[V(\{e,e',e''\})]$ contains no $M$-alternating cycles. By Lemma \ref{l1}, $M\setminus \{e,e',e''\}$ is a forcing set of $M$. So $f(G)\leq f(G,M)\leq n-3$, which is a contradiction.
 So the claim holds, and $G$ is not elementary by Lemma \ref{ll4}.

The subgraph of $G$ consisting of all allowed edges in $G$ and their end-vertices has components, say, $L_1, L_2,\dots,$ $L_k$ where $k\geq 2$.
Then two end-vertices of any forbidden edge of $G$ lie in different components.
If not, there exists a forbidden edge $e$ of $G$ whose two end-vertices belong to some component $L_i$. Let $L_i'$ be a graph obtained from $L_i$ by adding the edge $e$. Then $e$ is also a forbidden edge of $L_i'$, which contradicts Lemma \ref{ll4}. Hence all edges between distinct components are precisely forbidden edges of $G$.
 Thus,
\begin{eqnarray}
n-2=f(G)=\sum_{i=1}^kf(L_i)\leq \sum_{i=1}^k (\frac{|V(L_i)|}{2}-1)=\frac{1}{2}\sum_{i=1}^k|V(L_i)|-k=n-k,\label{ie2}
\end{eqnarray}
which implies that $k\leq 2$. So $k=2$ and all equalities in (\ref{ie2}) hold. So $f(L_i)=\frac{|V(L_i)|}{2}-1$ for $i=1$ and 2. By Theorem \ref{thm1}, $L_1$ and $L_2$ are two complete bipartite graphs. Hence $G$ is a graph in $\mathcal{G}_2$.
\end{proof}

Let $G$ be a bipartite graph of order $2n$ and with $F(G)=n-1$. By Theorem \ref{p7}, $e(G)= n^2$ and $G$ is $K_{n,n}$. Combining Theorem \ref{thm1}, we obtain the following result.
\begin{remark}{\rm Let $G$ be a bipartite graph of order $2n$ for $n\geq 2$. Then  $f(G)=n-2$ if and only if each perfect matching of $G$ has the forcing number $n-2$.}
\end{remark}
\begin{remark}\label{rm7} {\rm Let $G$ be a graph in $\mathcal{G}_1$ or $\mathcal{G}_2$. Then $G$ is disconnected if and only if $G$ is the disjoint union of two complete bipartite graphs with perfect matchings, i.e., there are exactly two deleted subgraphs and their orders are $n$.

It suffices to prove the necessity. Since $G$ is disconnected, $G$ has at least two components, say $L_1,L_2,\dots,L_k$ where $k\geq 2$. By Theorem \ref{thm4}, $f(G)=n-2$ and all equalities in (\ref{ie2}) hold. By the same arguments as the proof of Theorem \ref{thm4}, we obtain that $k=2$, and $L_1$ and $L_2$ are two complete bipartite graphs.}
\end{remark}
Next we will determine all elementary bipartite graphs in $\mathcal{G}_2$.
\begin{prop}Let $G$ be a graph in $\mathcal{G}_1$.
 Then $G$ is elementary if and only if each deleted subgraph of $G$ has order less than $n$.
\end{prop}
\begin{proof}Sufficiency. By Remark \ref{rm7}, $G$ is connected. For an edge $e$ of $G$, let $G'=G-V(e)$. Then $G'$ is a graph obtained from $K_{n-1,n-1}$ by deleting all edges of some  disjoint complete bipartite subgraphs and the orders of its deleted subgraphs
(if exists) are no more than $n-1$. So $\alpha(G')=n-1$. By Lemma \ref{ll3}, $G'$ has a perfect matching $M'$, and $M'\cup\{e\}$ is a perfect matching of $G$. Hence $e$ is allowed. By Lemma \ref{ll4}, $G$ is elementary.

Necessity. Since $G$ is elementary, it has a perfect matching. So each deleted  subgraph of $G$ has order no more than $n$. Suppose to the contrary that $K_{i,n-i}$ is a deleted subgraph of $G$ with $1\leq i\leq n-1$. Since $G$ is connected, the remaining $n$ vertices can not form another deleted subgraph of $G$ by Remark \ref{rm7}. So the orders of other deleted subgraphs of $G$ (if exists) are no more than $n-1$. Hence $G-V(K_{i,n-i})$ contains at least one edge, say $e$. Since $V(K_{i,n-i})$ forms an independent set of $G$ with cardinality $n$, $e$ is not allowed, which contradicts Lemma \ref{ll4}.
\end{proof}

\section{\normalsize Problems and conjectures}
Let $G$ be a graph of order $2n$ and with a perfect matching. By Theorem \ref{p7}, we obtain that $e(G)\geq \frac{n(n+1)}{n-F(G)}-F(G)-1$. But plenty of examples imply that this bound is not good enough. Since  $$\frac{n^2}{n-F(G)}-[\frac{n(n+1)}{n-F(G)}-F(G)-1]=\frac{F(G)[n-1-F(G)]}{n-F(G)}
\geq 0,$$ and equality holds if and only if $F(G)=0$ or $n-1$.
So we give a conjecture as follows.
\begin{conj}\label{conj}Let $G$ be a graph of order $2n$ and with a perfect matching. Then $e(G)\geq \frac{n^2}{n-F(G)}$. Equivalently, $F(G)\leq \frac{ne(G)-n^2}{e(G)}$.
\end{conj}
There are some examples showing that Conjecture \ref{conj} holds.
\begin{prop}For $F(G)\leq \frac{n}{2}$, Conjecture \ref{conj} holds.
\end{prop}
\begin{proof}Since $F(G)\leq \frac{n}{2}$, we have $\frac{F(G)^2}{n-F(G)}\leq F(G)$. So $\lceil\frac{F(G)^2}{n-F(G)}\rceil\leq F(G)$. By Proposition \ref{p4}, we have $F(G)\leq \frac{e(G)-n}{2}$. So $e(G)\geq n+2F(G)\geq n+F(G)+\lceil\frac{F(G)^2}{n-F(G)}\rceil=\lceil\frac{n^2}{n-F(G)}\rceil$.
\end{proof}

\begin{prop}\label{ex1}Let $G$ be a graph of order $2n$. If $F(G)=n-1$ or $n-2$,  then Conjecture \ref{conj} holds.
\end{prop}
\begin{proof}For $F(G)=n-1$, two bounds in Conjecture \ref{conj} and Theorem \ref{p7} are equal. So Conjecture \ref{conj} holds.

For $F(G)=n-2$, we will proceed by induction on $n$.
For $n=2$, we have $F(G)=0$. So $G$ has a unique perfect matching and $e(G)\geq 2$. Suppose that $n\geq 3$. Since $F(G)=n-2$, there exists a perfect matching $M$ of $G$ such that $f(G,M)=n-2$. By Lemma \ref{c2}, there exists an edge $uv\in M$ such that $d_G(u)+d_G(v)\geq n$. Let $G'=G-\{u,v\}$. Then $n-3\leq F(G')\leq n-2$.

If $F(G')=n-2=(n-1)-1$, then $e(G')\geq(n-1)^2$ by Theorem \ref{p7}. So $$e(G)=e(G')+d_G(u)+d_G(v)-1\geq(n-1)^2+n-1\geq
\frac{n^2}{2}.$$ Otherwise, we obtain that $F(G')=n-3=(n-1)-2$.  By the induction hypothesis, $e(G')\geq \lceil\frac{(n-1)^2}{n-1-(n-3)}\rceil=\lceil\frac{(n-1)^2}{2}\rceil$.
By Lemma \ref{c2}, we obtain that $d_G(u)+d_G(v)\geq n+1$ when $n$ is odd and $d_G(u)+d_G(v)\geq n$ when $n$ is even. Hence we have
\begin{equation*}
 e(G)=e(G')+d_G(u)+d_G(v)-1\geq
 \begin{cases}
 \lceil\frac{(n-1)^2}{2}\rceil+
(n+1)-1\geq \frac{n^2+1}{2}, & \quad {\text{if $n$ is odd}};\\
 \lceil\frac{(n-1)^2}{2}\rceil+n-1=\frac{n^2}{2},&\quad {\text{otherwise}}.
 \end{cases}
 \end{equation*}
Here we complete the proof.
\end{proof}
In Theorem \ref{thm4}, we have completely characterized all bipartite graphs $G$ of order $2n$ and with  $f(G)= n-2$. Here we propose the following problem.
\begin{pb}Determine all non-bipartite graphs $G$ of order $2n$ and with $f(G)=n-2$.
\end{pb}

For general 2-connected plane bipartite graphs, Abeledo and Atkinson \cite{13} obtained that the resonant number can be computed in polynomial time.
Hence the maximum forcing numbers of hexagonal systems \cite {27}, polyomino graphs \cite{42} and BN-fullerene graphs \cite{40} can be computed in polynomial time.
\vspace{2mm}

Afshani \cite{5} proposed a problem which has not been solved yet.
\begin{pb}\cite{5} What is the computational complexity of the maximum forcing numbers of graphs ?
\end{pb}



\end{document}